\documentclass[a4paper]{article}
\usepackage[margin=30mm]{geometry}
\usepackage{amsmath}
\usepackage{amssymb}
\usepackage{amsthm}
\usepackage{braket}
\usepackage{titlesec}
\usepackage{titlefoot}
\usepackage{comment}
\usepackage{graphicx}
\usepackage{bm}
\titleformat*{\section}{\normalsize\bfseries}
\titleformat*{\subsection}{\normalsize\bfseries}
\allowdisplaybreaks

\newtheorem{thm}{Theorem}

\newtheorem{prop}[thm]{Proposition}

 
\begin{document}

\title{
\vspace{-1.25cm}
\Large{\bf Mathematical and Numerical Analysis for Some \\
Nonlinear Second-order Ordinary Differential \\
Equations of Duffing Type}}
\author{Yusuke Kunimoto and Ikki Fukuda
}
\date{}
\maketitle

\footnote[0]{This manuscript is the preprint version of the original paper published in ``The 43rd JSST Annual International Conference on Simulation Technology, Conference Proceedings (2024) 494--501''.}

\vspace{-1cm}
\begin{abstract}
In this paper, we consider the initial value problem for some nonlinear second-order ODEs of Duffing type. We study the large time behavior of the solutions to this problem, from both the perspectives of mathematical and numerical analysis. First, we derive the decay estimate of the solutions, by using the energy method. Moreover, we numerically investigate the large time behavior of the energy function related to this problem, by using a structure-preserving difference method.
\end{abstract}


\section{Introduction}  

In this paper, we consider the initial value problem for the following nonlinear second-order ordinary differential equations of Duffing type: 
\begin{align}\label{Duffing}
\begin{cases}
x''(t)+\mu x'(t)+\alpha x(t)^p=0, \ \ t>0, \\
x(0)=x_0, \ x'(0)=x_1, \ \ (x_0, x_1) \in \mathbb{R}^2, 
\end{cases}
\end{align}
where $p\ge3$ is an odd number, $\mu\ge 0$ and $\alpha>0$. 
This equation \eqref{Duffing} is known as a mathematical model for describing nonlinear oscillation phenomena. 
For example, suppose we introduce a nonlinear spring for which Hooke's law does not apply and consider the motion of a point mass at the position $x(t)$ with a restoring force $\alpha x(t)^{p}$ under the effect of the damping $\mu x'(t)$. 
Then, Newton's second law leads \eqref{Duffing}. 
The typical example of \eqref{Duffing} is the case of $p=3$, which may be regarded as the special case $k=0$ of the following Duffing equation: 
\begin{equation}\label{Duffing-original}
x''(t)+\mu x'(t)+kx(t)+\alpha x(t)^3=0, \ \ t>0. 
\end{equation}

The above Duffing equation \eqref{Duffing-original} is an important mathematical model which has several applications, not only in nonlinear springs, but also in other fields of science and engineering. 
Actually, \eqref{Duffing-original} appears in some nonlinear signal detections (e.g.~\cite{Let22,Zet17}) and a damped motion of an extensible elastic rod with hinged ends (cf. \cite{B73-1,B73-2,BC76}). Our problem \eqref{Duffing} is a generalization of \eqref{Duffing-original} with $k=0$ (we call it ``Duffing type equation'' here). 
Therefore, we believe that our study of \eqref{Duffing} can be applicable in various fields of science and engineering.

Before introducing our results, let us recall some mathematical known results for \eqref{Duffing}. 
In Nakao~\cite{N77} and Nakao--Hara~\cite{NH12}, it is shown that the global existence of the solutions to \eqref{Duffing} and its decay estimate. More precisely, if $\mu>0$, they derived the following decay estimate of the solution to \eqref{Duffing}: 
\begin{equation}\label{known-result}
E(t):=\frac{1}{2}x'(t)^{2}+\frac{\alpha}{p+1}x(t)^{p+1}\le C(1+t)^{-\frac{p+1}{p-1}}, \ \ t\ge0. 
\end{equation}
Here, we note that they actually treated the more general equations (for details, see the original papers \cite{N77,NH12}). 
Moreover, we remark that the above function $E(t)$ means the mechanical energy, since the first term and the second term of $E(t)$ are the kinetic energy and the potential energy, respectively. 

From \eqref{known-result}, we can conclude that the solution $x(t)$ to \eqref{Duffing} decays at the rate of $t^{-1/(p-1)}$. 
Moreover, for some related results to the decay estimate of the solution $x(t)$, we can also refer to \cite{BC76,GZ13,MN04}. 
In particular, Ball--Carr~\cite{BC76} succeeded to derive the asymptotic profile of the solution $x(t)$. 
Their result suggests that the decay rate $t^{-1/(p-1)}$ of the solution $x(t)$ is optimal. 
Furthermore, for the special case of $p=3$, a more detailed higher-order asymptotic expansion is also obtained in \cite{BC76}. 
From their result, if $p=3$, we can see that not only the decay rate $t^{-1/(p-1)}$ of the solution $x(t)$ but also the decay rate $t^{-2}$ of the energy $E(t)$ is optimal. 

In this paper, we would like to study the large time behavior of the solution to \eqref{Duffing}, from both the perspectives of mathematical and numerical analysis. 
As we mentioned in the above, the decay estimate of the solution has already been obtained for $\mu>0$. 
However, we are able to show that the decay estimate of $x(t)$ can be obtained by a much simpler method than used in the previous papers \cite{BC76,N77,NH12}. 
Actually, by using the energy method, we succeeded to obtain a decay estimate which has the same decay rate $t^{-1/(p-1)}$. 
We note that our method requires only basic calculus. 
Next, we shall numerically investigate the large time behavior of the energy $E(t)$. 
As we have already seen in \eqref{known-result}, $E(t)$ decays at the rate of $t^{-(p+1)/(p-1)}$ if $\mu>0$. 
By virtue of this estimate, we can see the effect of changes in the exponent $p$ on the large time behavior of $E(t)$. 
However, from this result, we cannot understand how changing the damping coefficient $\mu$ or the restoring coefficient $\alpha$ affects its behavior. 
Thus, we shall numerically analyze $E(t)$ from this perspective. 
In order to do that, we need to perform numerical simulations which preserve the structure of the energy. 
Therefore, we shall use the structure-preserving numerical method used in \cite{FM10}. 

\section{Mathematical Analysis}

In this section, we would like to state our mathematical results about \eqref{Duffing}, which related to the decay estimate for the solution $x(t)$. The main result is as follows:  
\begin{thm}\label{main-decay}
Let $p\ge3$ be an odd number, $\mu\ge0$ and $\alpha>0$. Then, for any initial data $(x_{0}, x_{1})\in \mathbb{R}^{2}$, 
there exists a unique global solution $x\in C^{2}([0, \infty))$ to \eqref{Duffing}. Moreover, if $\mu>0$, there exists a positive constant $C_{*}>0$ such that the solution $x(t)$ satisfies the following estimate: 
\begin{equation}\label{sol-decay}
\left|x(t)\right| \le C_{*}t^{-\frac{1}{p-1}}, \ \ t>0. 
\end{equation}
\end{thm}

In order to analyze \eqref{Duffing}, let us follow the method used in \cite{MN04}. 
First, we shall rewrite \eqref{Duffing} as the following system of first-order differential equations:  
\begin{align}\label{eq:IVP-1} 
\begin{cases}
x'(t)=y(t), \\
y'(t)=-\alpha x(t)^p-\mu y(t), \ \ t>0, \\
x(0)=x_0, \ y(0)=y_0, \ \ (x_0,y_0)\in \mathbb{R}^2,  
\end{cases}
\end{align}
where $y_0:=x_1$. From the general theory of ordinary differential equations, we can easily prove that there exists a positive constant $T>0$ and a unique local solution $(x, y)\in C^{1}([0, T]; \mathbb{R}^{2})$ to \eqref{eq:IVP-1}. 
Moreover, this local solution can be easily extended globally in time. 
Actually, multiplying $y(t)$ on the both sides of the second equation in \eqref{eq:IVP-1}, we have
\begin{align}
\left(\dfrac{1}{2}y(t)^2+\dfrac{\alpha}{p+1}x(t)^{p+1}\right)'+\mu y(t)^2=0. \label{eq:E1}
\end{align}
Therefore, integrating the both sides of the above result on $[0, t]$, we obtain the following a priori estimate for the solution: 
\begin{align}\label{eq:m-energy} 
E(t)+\mu \int_{0}^{t}y(\tau)^2d\tau=E(0), \ \ t\in [0,T], 
\end{align}
where $E(t)$ is defined in \eqref{known-result}. Therefore, we can continue the local solution and prove the global existence of the solution $(x, y)\in C^{1}([0, \infty); \mathbb{R}^{2})$ to \eqref{eq:IVP-1}, i.e. there exists a unique global solution $x\in C^{2}([0, \infty))$ to \eqref{Duffing} (for details, see e.g. \cite{MN04}). Moreover, we note that \eqref{eq:m-energy} means the energy conservation law if $\mu=0$ and the dissipation law if $\mu>0$, because $E(t)$ represents the mechanical energy.

By virtue of the above discussions, it is sufficient to prove only the decay estimate \eqref{sol-decay}. 
In order to do that, we need to prepare two propositions. First, multiplying $\mu x(t)/2$ on the both sides of the second equation in \eqref{eq:IVP-1}, we obtain 
\begin{align}
\left( \dfrac{\mu}{2}x(t) y(t) + \dfrac{\mu^2}{4}x(t)^2 \right)' +\dfrac{\mu\alpha}{2} x(t)^{p+1} - \dfrac{\mu}{2} y(t)^2 =0.  \label{eq:E2} 
\end{align}
Therefore, combining \eqref{eq:E1} and \eqref{eq:E2}, we have the following relation: 
\begin{align}\label{eq:energy-simple} 
\mathcal{E}'(t)+\mathcal{H}(t) = 0, 
\end{align}
where the functions $\mathcal{E}(t)$ and $\mathcal{H}(t)$ are defined by 
\begin{align}\label{energy-functional}
\begin{split}
&\mathcal{E}(t):=\frac{1}{2}y(t)^2+\frac{\mu}{2}x(t)y(t)+\frac{\mu^2}{4}x(t)^2+\frac{\alpha}{p+1}x(t)^{p+1},  \\
&\mathcal{H}(t):=\frac{\mu}{2}y(t)^2+\frac{\mu\alpha}{2}x(t)^{p+1}. 
\end{split}
\end{align}
In what follows, we shall derive a differential inequality and a decay estimate for $\mathcal{E}(t)$. Actually, in the case of $p=3$, Propositions~\ref{energy-d-ineq} and \ref{energy-decay} below are shown in \cite{MN04}. Our results are one of the generalization of them. Now, let us state such results. We start with introducing the following differential inequality: 
\begin{prop}\label{energy-d-ineq}
Let $p\ge3$ be an odd number, $\mu>0$ and $\alpha>0$. Then, there exists a positive constant $\nu>0$ such that $\mathcal{E}(t)$ satisfies the following inequality:  
\begin{align}\label{eq:d-ineq} 
\mathcal{E}'(t)+\nu\dfrac{\mathcal{E}(t)^{\frac{p+1}{2}}}{1+\mathcal{E}(t)^{\frac{p-1}{2}}}\le 0, \ \ t\ge 0. 
\end{align} 
\end{prop}
\begin{proof}
First, we would like to evaluate $\mathcal{E}(t)$ from below as follows: 
\begin{align*}
\mathcal{E}(t)
=\frac{1}{4}\left(y(t)+\mu x(t)\right)^2 +\frac{1}{4}y(t)^2+\frac{\alpha}{p+1}x(t)^{p+1} \ge \frac{1}{4}y(t)^2+\frac{\alpha}{p+1}x(t)^{p+1}.
\end{align*}
Since the right hand side of the above is positive, we have 
\begin{align}
\mathcal{E}(t)^{\frac{p-1}{2}} 
\ge \left(\frac{y(t)}{2}\right)^{p-1} + \left(\frac{\alpha}{p+1}x(t)^{p+1}\right)^{\frac{p-1}{2}}. \label{eq:energy-lower-2} 
\end{align}
On the other hand, it follows from Young's inequality that 
\begin{align}
\mathcal{E}(t)
\le \frac{2+\mu}{4}y(t)^2 +\frac{\mu(1+\mu)}{4} x(t)^2 +\frac{\alpha}{p+1}x(t)^{p+1}. \label{eq:energy-upper}
\end{align}
Therefore, \eqref{eq:energy-upper}, \eqref{eq:energy-lower-2} and \eqref{energy-functional} yield
\begin{align}
\mathcal{E}(t)^{\frac{p+1}{2}} 
&\le C\left\{ y(t)^{p+1} +x(t)^{p+1} +x(t)^{\frac{(p+1)^2}{2}}\right\} \nonumber \\
&\le C\left\{y(t)^{p-1}\cdot \frac{\mu}{2}y(t)^2+\left(1+ x(t)^{\frac{p^2-1}{2}}\right) \cdot \frac{\mu \alpha}{2}x(t)^{p+1}\right\} \nonumber \\ 
&\le C\left\{1+\left(\frac{y(t)}{2}\right)^{p-1}+\left(\frac{\alpha}{p+1}x(t)^{p+1}\right)^{\frac{p-1}{2}}\right\}\mathcal{H}(t) \nonumber \\
&\le C\left\{1+\mathcal{E}(t)^{\frac{p-1}{2}}\right\} \mathcal{H}(t), \ \ t\ge0. \label{eq:d-ineq-pre}
\end{align}
Finally, combining \eqref{eq:d-ineq-pre} and \eqref{eq:energy-simple}, we can conclude that there exists a positive constant $\nu>0$ such that the desired inequality \eqref{eq:d-ineq} holds. 
\end{proof}

By virtue of the above result, we can derive the decay estimate for $\mathcal{E}(t)$: 
\begin{prop}\label{energy-decay}
Let $p\ge3$ be an odd number, $\mu>0$ and $\alpha>0$. Then, there exists a positive constant $C_{\dag}>0$ such that $\mathcal{E}(t)$ satisfies the following estimate:  
\begin{align}\label{eq:energy-decay} 
\mathcal{E}(t)\leq C_{\dag} t^{-\frac{2}{p-1}}, \ \ t>0. 
\end{align} 
\end{prop}
\begin{proof}
First, multiplying $e^{\nu t-\frac{2}{p-1}\mathcal{E}(t)^{-\frac{p-1}{2}}}$ on the both sides of the differential inequality \eqref{eq:d-ineq}, we are able to see that 
\begin{align*}
&e^{\nu t-\frac{2}{p-1}\mathcal{E}(t)^{-\frac{p-1}{2}}} \mathcal{E}'(t)
+\nu e^{\nu t-\frac{2}{p-1}\mathcal{E}(t)^{-\frac{p-1}{2}}}\frac{\mathcal{E}(t)^{\frac{p+1}{2}}}{1+\mathcal{E}(t)^{\frac{p-1}{2}}} \le 0 \\
\Longleftrightarrow \ \ &e^{\nu t-\frac{2}{p-1}\mathcal{E}(t)^{-\frac{p-1}{2}}} \{\mathcal{E}'(t) + \nu \mathcal{E}(t)^{\frac{p+1}{2}} \} + e^{\nu t-\frac{2}{p-1}\mathcal{E}(t)^{-\frac{p-1}{2}}} \mathcal{E}'(t) \mathcal{E}(t)^{\frac{p-1}{2}} \le 0 \\
\Longleftrightarrow \ \ &e^{\nu t-\frac{2}{p-1}\mathcal{E}(t)^{-\frac{p-1}{2}}}
\left\{\frac{\mathcal{E}'(t)}{\mathcal{E}(t)^{\frac{p+1}{2}}}+ \nu \right\} \mathcal{E}(t) 
+e^{\nu t-\frac{2}{p-1}\mathcal{E}(t)^{-\frac{p-1}{2}}} \mathcal{E}'(t) \le 0. 
\end{align*}
Therefore, $\mathcal{E}(t)$ satisfies the following inequality: 
\begin{align*}
\left\{e^{\nu t-\frac{2}{p-1}\mathcal{E}(t)^{-\frac{p-1}{2}}} \mathcal{E}(t)\right\}'\le 0, \ \ t\ge 0. 
\end{align*}
Integrating the both sides of the above result on $[0, t]$, we have  
\begin{align*}
e^{\nu t-\frac{2}{p-1}\mathcal{E}(t)^{-\frac{p-1}{2}}}\mathcal{E}(t)\leq e^{-\frac{2}{p-1}\mathcal{E}(0)^{-\frac{p-1}{2}}}\mathcal{E}(0)=:\mathcal{E}_{0}. 
\end{align*}
Taking the logarithm on the both sides of the above, and then multiplying $\mathcal{E}(t)^{\frac{p-1}{2}}$ on its result, 
we arrive at the following inequality: 
\begin{align}
\nu t \mathcal{E}(t)^{\frac{p-1}{2}}-\dfrac{2}{p-1}+\mathcal{E}(t)^{\frac{p-1}{2}}\log \mathcal{E}(t) 
\le \mathcal{E}(t)^{\frac{p-1}{2}}\log \mathcal{E}_{0}. \label{eq:up-low-pre}
\end{align}

In what follows, let us evaluate the both sides of \eqref{eq:up-low-pre}. 
First, we shall evaluate the right hand side from above. It follows from \eqref{eq:energy-simple} and $\mathcal{H}(t)\ge0$ that $\mathcal{E}'(t)\le 0$, and then $\mathcal{E}(t)\le \mathcal{E}(0)$ holds for all $t\ge0$. Therefore, we have 
\begin{align}\label{eq:energy-up} 
\mathcal{E}(t)^{\frac{p-1}{2}}\log\mathcal{E}_{0} 
\leq 
\mathcal{E}(0)^{\frac{p-1}{2}}\left|\log\mathcal{E}_{0}\right|. 
\end{align}
Next, we would like to deal with the left hand side of \eqref{eq:up-low-pre}. 
Now, we shall use a fact that $f(s):=s^{\frac{p-1}{2}}\log{s}\geq -\dfrac{2}{e(p-1)}$ holds for all $s>0$. By virtue of this fact, substituting $s=\mathcal{E}(t)$ into the above $f(s)$, we obtain 
\begin{align*}
\mathcal{E}(t)^{\frac{p-1}{2}}\log{\mathcal{E}(t)}\ge -\frac{2}{e(p-1)}. 
\end{align*}
Thus, we can evaluate the left hand side of \eqref{eq:up-low-pre} from below  as follows: 
\begin{align} \label{eq:energy-low} 
\nu t \mathcal{E}(t)^{\frac{p-1}{2}}-\dfrac{2}{p-1}+\mathcal{E}(t)^{\frac{p-1}{2}}\log{\mathcal{E}(t)} 
\geq 
\nu t \mathcal{E}(t)^{\frac{p-1}{2}}-\dfrac{2(e+1)}{e(p-1)}. 
\end{align}

Finally, it follows from \eqref{eq:up-low-pre}, \eqref{eq:energy-up} and \eqref{eq:energy-low} that 
\begin{align*}
\nu t \mathcal{E}(t)^{\frac{p-1}{2}}-\dfrac{2(e+1)}{e(p-1)}
\le \mathcal{E}(0)^{\frac{p-1}{2}}\left|\log\mathcal{E}_{0}\right|. 
\end{align*}
Therefore, we can eventually see that there exists positive constant $C_{\dag}>0$ such that $\mathcal{E}(t)\leq C_{\dag} t^{-\frac{2}{p-1}}$ holds for all $t>0$. This completes the proof of the desired estimate \eqref{eq:energy-decay}. 
\end{proof}

In the rest of this section, we shall give the proof of the decay estimate \eqref{sol-decay}:  
\medskip

\noindent{\bf End of the proof of Theorem~\ref{main-decay}.} 
From the definition of $\mathcal{E}(t)$ in \eqref{energy-functional}, we have  
\begin{align*}
\mathcal{E}(t)
=\frac{1}{2}\left(y(t)+\frac{\mu}{2}x(t)\right)^2 +\frac{\mu^{2}}{8}x(t)^2+\frac{\alpha}{p+1}x(t)^{p+1}\ge \frac{\mu^{2}}{8}x(t)^2. 
\end{align*}
Thus, by virtue of \eqref{eq:energy-decay}, we can conclude that the desired estimate \eqref{sol-decay} is true. \qed

\section{Numerical Analysis}

In this section, let us treat the numerical analysis for \eqref{Duffing}, by using the structure-preserving numerical method used in \cite{FM10}. 
Now, we consider \eqref{eq:IVP-1} on the interval $[0, T]$ with sufficiently large $T>0$. Let $N\in \mathbb{N}$ and split the interval $[0, T ]$ into $N$-th parts with the time mesh size $\Delta t$, and hence the relation $T=N\Delta t$ holds. 
In finite difference method, we denote a solution to the corresponding discrete problem at $t=n\Delta t$ ($n = 0,1, \cdots, N$) by $(x^{(n)}, y^{(n)})$. 
Also, for an approximation to derivatives $x'(t)$ and $y'(t)$, let us define the difference operator: 
\[\displaystyle \delta_{n}^{+}f^{(n)}:=\frac{f^{(n+1)}-f^{(n)}}{\Delta t}. \]
Then, we would like to follow the technique used in \cite{FM10} and construct the following structure-preserving difference scheme for \eqref{eq:IVP-1} as follows:  
\begin{align}\label{SP} 
\begin{cases}
\displaystyle \delta_{n}^{+}x^{(n)}=\frac{1}{2}\left(y^{(n+1)}+y^{(n)}\right), \\
\displaystyle \delta_{n}^{+}y^{(n)}=-\frac{\alpha}{p+1}\sum_{l=0}^{p}\left(x^{(n+1)}\right)^{p-l}\left(x^{(n)}\right)^{l}-\frac{\mu}{2}\left(y^{(n+1)}+y^{(n)}\right). 
\end{cases}
\end{align}

Here, let us mention that a discrete version of \eqref{eq:m-energy} holds for the above scheme \eqref{SP}. 
First, we define the corresponding discrete energy as follows: 
\begin{equation*}
E^{(n)}:=\frac{1}{2}\left(y^{(n)}\right)^{2}+\frac{\alpha}{p+1}\left(x^{(n)}\right)^{p+1}. 
\end{equation*}
Multiplying $\left(y^{(n+1)}+y^{(n)}\right)/2$ on the both sides of the second equation in \eqref{SP} and using a fact $a^{p+1}-b^{p+1}=(a-b)(a^{p}+a^{p-1}b+\cdots +ab^{p-1}+b^{p})$, we have
\begin{align*}
\frac{E^{(n+1)}-E^{(n)}}{\Delta t}+\frac{\mu}{4}\left(y^{(n+1)}+y^{(n)}\right)^{2}=0. 
\end{align*}
Therefore, the discrete energy $E^{(n)}$ satisfies the energy conservation law and the dissipation law corresponding to \eqref{eq:m-energy}, in the following sense: 
\begin{equation*}
E^{(n)}+\frac{\mu}{4}\sum_{l=0}^{n-1}\left(y^{(l+1)}+y^{(l)}\right)^{2} \Delta t
=E^{(0)}, \ \ n = 1, 2, \cdots, N. 
\end{equation*}

For basic literatures of the structure-preserving numerical method, let us refer to \cite{FM10} and also references therein. Moreover, for the more mathematical topics about the structure-preserving numerical method related to nonlinear differential equations, such as the global existence of the solution to \eqref{SP} or its error estimates, are well summarized in recent papers by Yoshikawa~\cite{Y17,Y19}. Thus, for the theoretical parts of \eqref{SP}, we shall leave them to these papers. 

\bigskip
\noindent
\textbf{\bf{Numerical Results}} 

\medskip

Finally, we shall explain our numerical simulations for $E(t)$. 
We set the data as $(x_0, y_0)=(2, 0)$ and performed the calculations for $T=5000$ with $\Delta t = 0.01$. The following Figure~\ref{result} shows the simulation results for $E(t)$ corresponding to $\mu=0$, $\mu=0.1$, $\mu=1$, $\mu=10$ and $\mu=100$, for $(p, \alpha)=(3, 1)$, $(p, \alpha)=(3, 100)$, $(p, \alpha)=(5, 1)$, $(p, \alpha)=(5, 100)$, $(p, \alpha)=(7, 1)$ and $(p, \alpha)=(7, 100)$. 
The decay rate $t^{-(p+1)/(p-1)}$ appeared in \eqref{known-result} is also shown in the figure. 
First, we can confirm that the conservation law \eqref{eq:m-energy} holds when $\mu=0$ for all cases. 
When $\mu>0$, contrary to our intuition, we can see that, the larger $\mu$ is, the greater the energy, for $t\gg1$. 
Also, comparing the cases of $\alpha=1$ and $\alpha=100$, it can be seen that the energy is smaller in the case of $\alpha=100$. 

\begin{figure}
\centering
\includegraphics[width=\linewidth]{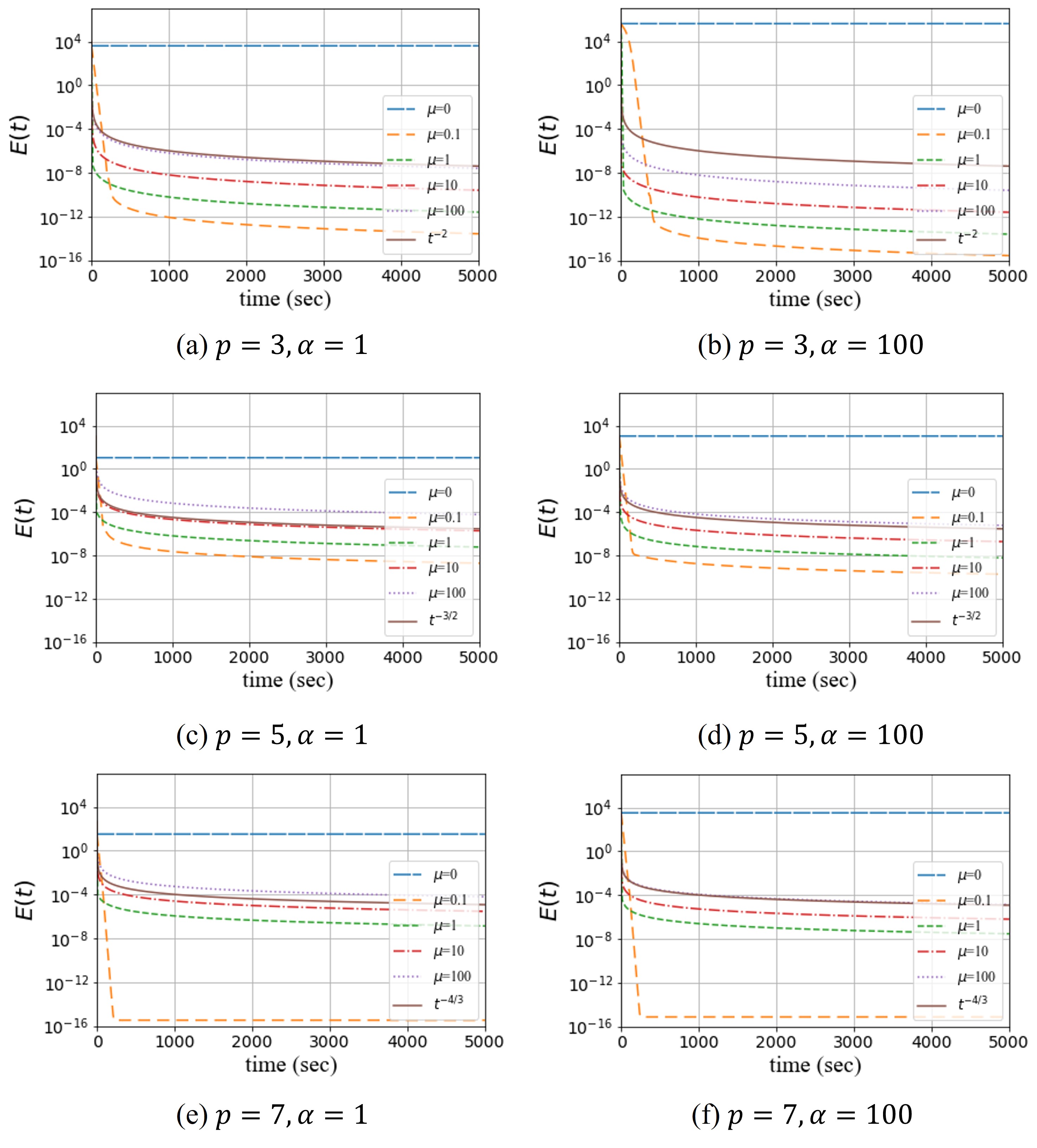}
\caption{Numerical results for $E(t)$ by using a structure-preserving numerical schemes.}\label{result}
\end{figure}

\newpage

\section*{Acknowledgments}

This study is supported by Grant-in-Aid for Young Scientists Research No.22K13939, Japan Society for the Promotion of Science. 

\section*{Disclosure of Interests}

The authors have no competing interests to declare that are relevant to the content of this article. 



\medskip
\par\noindent
\begin{flushleft}
Yusuke Kunimoto\\
Department of Civil Engineering, Kyushu University, \\
Fukuoka, 819-0395 Japan, \\
kunimoto.yusuke.222@s.kyushu-u.ac.jp

\bigskip
\par\noindent
Ikki Fukuda\\
Faculty of Engineering, Shinshu University, \\
Nagano, 380-8553, Japan\\
E-mail: i\_fukuda@shinshu-u.ac.jp
\end{flushleft}


\begin{thebibliography}{99}
\addcontentsline{toc}{section}{References}

\bibitem{B73-1}
Ball, J.M.: Initial-boundary value problems for an extensible beam. J. Math. Anal. Appl. {\bf 42}, 61--90 (1973).

\bibitem{B73-2}
Ball, J.M.: Stability theory for an extensible beam. J. Differ. Equ. {\bf 14}, 339--418 (1973).

\bibitem{BC76}
Ball, J.M., Carr, J.: Decay to zero in critical cases of second order ordinary differential equations of Duffing type. 
Arch. Rational Mech. Anal. {\bf 63}, 47--57 (1976).

\bibitem{FM10}
Furihata, D., Matsuo, T.: Discrete variational derivative method. Numerical Analysis and Scientific Computing Series, CRC Press/Taylor and Francis (2010).

\bibitem{GZ13}
Grushkovskaya, V., Zuyev, A.: Asymptotic behavior of solutions of a nonlinear system in the critical case of $q$ pairs of purely imaginary eigenvalues. Nonlinear Anal. {\bf 80}, 156--178 (2013).  

\bibitem{Let22}
Li, G., Hou, Y., Yang, H.: A new Duffing detection method for underwater weak target signal. Alexandria Engineering Journal {\bf 61}, 2859--2876 (2022).  

\bibitem{MN04}
Matsumura, A., Nishihara, K.: Global solutions of nonlinear differential equations--Mathematical analysis for compressible viscous fluids--. Nippon-Hyoron-Sha, Tokyo (2004) (in Japanese).

\bibitem{N77}
Nakao, M.: Decay of solutions of some nonlinear evolution equations. J. Math. Anal. Appl. {\bf 60}, 542--549 (1977).

\bibitem{NH12}
Nakao, M., Hara, A.: Global existence and decay for some nonlinear second order ordinary differential equations of degenerate type. Scientiae Mathematicae Japonicae {\bf 75}, 95--103 (2012).

\bibitem{Y17}
Yoshikawa, S.: Energy method for structure-preserving finite difference schemes and some properties of difference quotient. J. Comput. Appl. Math. {\bf 311}, 394--413 (2017).

\bibitem{Y19}
Yoshikawa, S.: Remarks on energy methods for structure-preserving finite difference schemes - Small data global existence and unconditional error estimate. Appl. Math. Comput. {\bf 341}, 80--92 (2019).

\bibitem{Zet17}
Zhang, Y.H., Mao, H.L., Mao, H.Y., Huang, Z.F.: Detection the nonlinear ultrasonic signals based on modified Duffing equations. Results Phys. {\bf 7}, 3243--3250 (2017).

\end{thebibliography}
\end{document}